\documentclass[makeidx,reqno]{amsart}
\usepackage[active]{srcltx}
\usepackage{epsfig}
\usepackage{hyperref}
\usepackage{amsmath}
\usepackage{amssymb}
\newtheorem{Proposition}{Proposition}
  \newtheorem{Remark}{Remark}
  \newtheorem{Corollary}[Proposition]{Corollary}
  \newtheorem{Lemma}[Proposition]{Lemma}
  
  \newtheorem{Theorem}{Theorem}

\newtheorem{Example}[Proposition]{Example}
\newtheorem{Note}[Proposition]{Note}
  
\newtheorem{Assumptions}{Assumption}

\newcommand {\z}{{\noindent}}
 
\def\CC{\mathbb{C}}
 
 \def\NN{\mathbb{N}}

\def\QQ{\mathbb{Q}}
\def\Re{\mathrm{Re}}
\def\Im{\mathrm{Im}}

\def\phi{\varphi}

 \def\({\left(} \def\){\right)} \makeindex
\author{O. Costin, M. Huang} \address{Mathematics Department, The Ohio State University, Columbus, OH 43210} \title{On the geometry of Julia sets} 
\begin{document}
\begin{abstract}
   We show that the Julia set of quadratic maps with parameters in
  hyperbolic components of the Mandelbrot set is given by a transseries formula, rapidly convergent  at any repelling  periodic point.

Up to conformal transformations, we obtain $J$
  from a smoother curve of lower Hausdorff
  dimension, by replacing pieces of the more regular curve by
  increasingly rescaled elementary ``bricks'' obtained from the
  transseries expression. Self-similarity of $J$, 
up to conformal transformation, is manifest in the formulas. 

The Hausdorff dimension of $J$ is estimated by the transseries formula.
The analysis extends  to polynomial maps.

\end{abstract}
\maketitle
\tableofcontents
\section{Introduction}
 Iterations of maps  are among the simplest mathematical models
of chaos. The understanding
of their behavior and of the associated fractal 
Julia sets (cf. \S\ref{KR} for a brief overview of relevant notions)  has progressed tremendously since the pioneering
work of Fatou and Julia (cf. \cite{Fatou1}--\cite{Julia}). The subject triggered the
development of powerful new mathematical tools. In spite of a vast literature  and of a century 
of fundamental advances, many important
problems are still not completely understood, even in the case of quadratic
maps, see {\em e.g.} \cite{Milnor}.

As discussed in \cite{Collet}, a ``major open question about fractal sets is to
provide quantities which describe their complexity.  The Hausdorff
dimension is the most well known such quantity, but it does not tell
much about the fine structure of the set. Two sets with the same
Hausdorff dimension may indeed look very different (computer pictures
of these sets, although very approximate, may reveal such
differences).``

A central goal of this paper is to provide a detailed geometric analysis
of local properties of Julia sets of polynomial maps.

It will be apparent from the proof that the method and
many results apply to polynomials of any order. However, we will frequently use for illustration purposes the quadratic map
$z\mapsto z^2+c$, or equivalently after a linear change of variable,
$x\mapsto P(x)=\lambda x(1-x)$, $c=\lambda/2-\lambda^2/4$ (note the symmetry $\lambda\to 2-\lambda$). The associated map 
iteration is
\begin{equation}
  \label{eq:logist11}
  x_{n+1}=P(x_n)
\end{equation}

Our analysis applies to the hyperbolic components of the Mandelbrot set
of the iteration, see \S\ref{KR}.
Let $\varphi$ be the B\"otcher map of $P$ (with the definition \eqref{eq:eqvphi} below), analytic 
in the punctured unit disk $\mathbb{D}\setminus \{0\}$; then $J=\varphi(\partial\mathbb{D})$. 

Let
 $p_0=\varphi(e^{2\pi i t_0})$ be any periodic point on $J$. We show that for $z\in\mathbb{D}$ 
near $e^{2\pi i t_0}$,
$\varphi$  is given by an entire function of
$s^b\omega(\ln s)$ where $s=-\ln( e^{-2\pi i t_0} z)$ (clearly
$s\to 0$ as $p_0$ is approached). Here $b$ has a simple formula and
$\omega$ is a real analytic periodic function.

In particular, at any such $p_0$, the local shape of $J$ is, to leading order,
the image of the segment $[-\epsilon,\epsilon]$ under a map of
the form $Az^b$. The averaged value $b_E$  of $b$  (over
all periodic points)
and the Hausdorff dimension of $J$, $D_H$, satisfy
$D_H\ge 1/\Re\, b_E$.
Up to conformal transformations, $J$ is obtained from a curve of
higher regularity and lower Hausdorff dimension than $J$, by replacing
pieces of the more regular curve by increasingly rescaled basic
``bricks'' obtained from the transseries expression. This is analogous
to the way elementary fractals such as the Koch snowflake are obtained.
We present figures representing $J$ for various
parameters in this constructive way.

\subsection{Notation and summary of known results}\label{KR}
  We use
standard terminology from the theory of iterations of holomorphic
maps.  If $f$ is an entire map, the set of points in $\CC$ for which
the iterates $\{f_n\}_{n\in\NN}$ form a normal family in the sense of
Montel is called the Fatou set $\mathcal{F}$ of $f$ while the Julia
set $J$ is $\CC\setminus \mathcal{F}$. For instance, if $f=P$ ($P$ being the quadratic map above)
and
$|\lambda|<1$, the Fatou set has two connected components,
$\mathcal{F}_{\infty}=\{z_0: |P_n(z_0)|\to \infty\text{ as }
n\to\infty\}$ and $\mathcal{F}_{0}=\{z_0: |P_n(z_0)|\to 0\text{ as }
n\to\infty\}$ where
$$P_n=P^{\circ n}$$
is the $n$-th self-composition of $P$. The Julia set in this example
is $\overline{\mathcal{F}_{0}}\cap \overline{\mathcal{F}_{\infty}}$, a
Jordan curve.

The substitution $x=-y^{-1}$ transforms (\ref{eq:logist11}) into
\begin{equation}
  \label{eq:infty1}
  y_{n+1}=\frac{y_n^2}{\lambda(1+ y_n)}=f(y_n)
\end{equation}

By B\"otcher's theorem (for (\ref{eq:logist11}) we give a self contained
proof in \S\ref{S5}), there exists a unique  map $F$, analytic near zero,
with $F(0)=0$, $F'(0)=\lambda^{-1}$ so that $(F\circ f\circ F^{-1})(x)=x^2$. Its inverse, $G$, used in \cite{DouadyHubbard}, conjugates (\ref{eq:infty1}) to the canonical map $z_{n+1}=z_n^2$, and it can be checked that
\begin{equation}
  \label{eq:eqG}
  G(z)^2=\lambda G(z^2)(1+G(z));\ \ G(0)=0,\ G'(0)=\lambda
\end{equation}
Equivalently (with $\varphi=1/G$) there exists a unique map $\varphi$
analytic in $\mathbb{D}\setminus \{0\}$
\begin{equation}
  \label{eq:eqvphi}
  P(\varphi(z))=\varphi(z^2);\ \ \ z\in \mathbb{D}\setminus \{0\};\ \ \lim_{z\to 0}z\varphi(z)=1/\lambda
\end{equation}
We list some further definitions and  known facts that we use; see, e.g., \cite{Beardon,Devaney,Milnor}.
\begin{Remark}\label{R7}{\rm
    \begin{enumerate}

\item $J$ is the closure of the set of repelling periodic points.

\item For polynomial maps, and more generally, for entire maps, $J$ is the boundary of the set of points which converge to infinity under the iteration.

\item For general polynomial maps, if the maximal disk of analyticity of $G$
(where now $1/G(z^n)=P_n(1/G(z))$) is the unit disk $\mathbb{D}_1$, then $G$ maps $\mathbb{D}_1$ biholomorphically onto the immediate basin $\mathcal{A}_0$ of zero.  If on the contrary the maximal disk is $\mathbb{D}_r,\,r<1$, then there is at least one other critical point in $\mathcal{A}_0$, lying in $G(\partial\mathbb{D}_r)=J_y$, the Julia set of (\ref{eq:infty1}), see \cite{Milnor} p.93.

\item If $r=1$, it follows that $G(\partial\mathbb{D}_1)=J_y$.

\item  For the iteration $t_{n+1}=t_n^2+c$, the  Mandelbrot set is defined
as (see e.g. \cite{Devaney})
\begin{equation}
  \label{eq:defM}
  \mathcal{M}=\{c: t_n \text{ bounded }\text{ if } t_0=0\}
\end{equation}
If $c\in\mathcal{M}$, then clearly $y_n$ in  (\ref{eq:infty1}) are bounded away from zero.
\item \label{(iii)}  $\mathcal{M}$ is a compact set; it coincides
with the set of $c$ for which $J$ is connected. The main
cardioid $\mathcal{H}=\{(2e^{it}-e^{2it})/4:t\in [0,2\pi)\}$ is contained in
$\mathcal{M}$; see \cite{Devaney}. This means $\{\lambda:|\lambda|<1\}$
corresponds to the interior of $\mathcal{M}$. We have
 $|\lambda|=1\Rightarrow c\in \partial \mathcal{M}\subset\mathcal{M}$.

Assume now that $c\in\mathcal{M}$ and $\lambda\ne 0$. Then,

\item The function $\varphi$ extends analytically to $\mathbb{D}$.

  \item (\cite{DouadyHubbard} p. 121)  If  $z$ approaches a rational angle, $e^{2\pi i t},\
t\in\QQ$, then the limit
\begin{equation}
  \label{eq:eqH0}
  L_t=\lim_{\rho\to 1} \varphi\left(\rho e^{2\pi i t}\right)\ \ \text{exists.}
\end{equation}
\item  (\cite{Milnorp})  A quadratic map has at most one non-repelling periodic orbit.

\item
   For every $\lambda$ such that the corresponding $c$ is in $\mathcal{M}$,
and any $t\in\QQ$, the limit
\begin{equation}
  \label{eq:eqH}
  \varphi\left(z e^{2\pi i t}\right)\to  L_t  \ \text{as}\ z\to 1 \ \text{nontangentially}  \ \footnote{Nontangentially is understood, as usually, as $z\to 1$  along any
          curve inside $\mathbb{D}$ which   lies between two
straight line segments inside $\mathbb{D}$ not tangent to
$\partial \mathbb{D}$.}
\end{equation}
exists.
(See also \cite{McMullen}.)
  This follows immediately from (\ref{eq:eqH0}), the boundedness of $\varphi$
and the Sectorial Limit Theorem, see \cite{Conway}, p. 23, Theorem 5.4.

\item In any hyperbolic component of $\mathcal{M}$ (components of
$\mathcal{M}$ corresponding to (unique) attracting cycles), the points $z\in$ fix $P_n$
on  the corresponding Julia set have the property $|P'_n(z)|>1$ and $\varphi$
is continuous in $\overline{\mathbb{D}}$.
\end{enumerate} }

\begin{figure}
  \centering
 \includegraphics[scale=0.4]{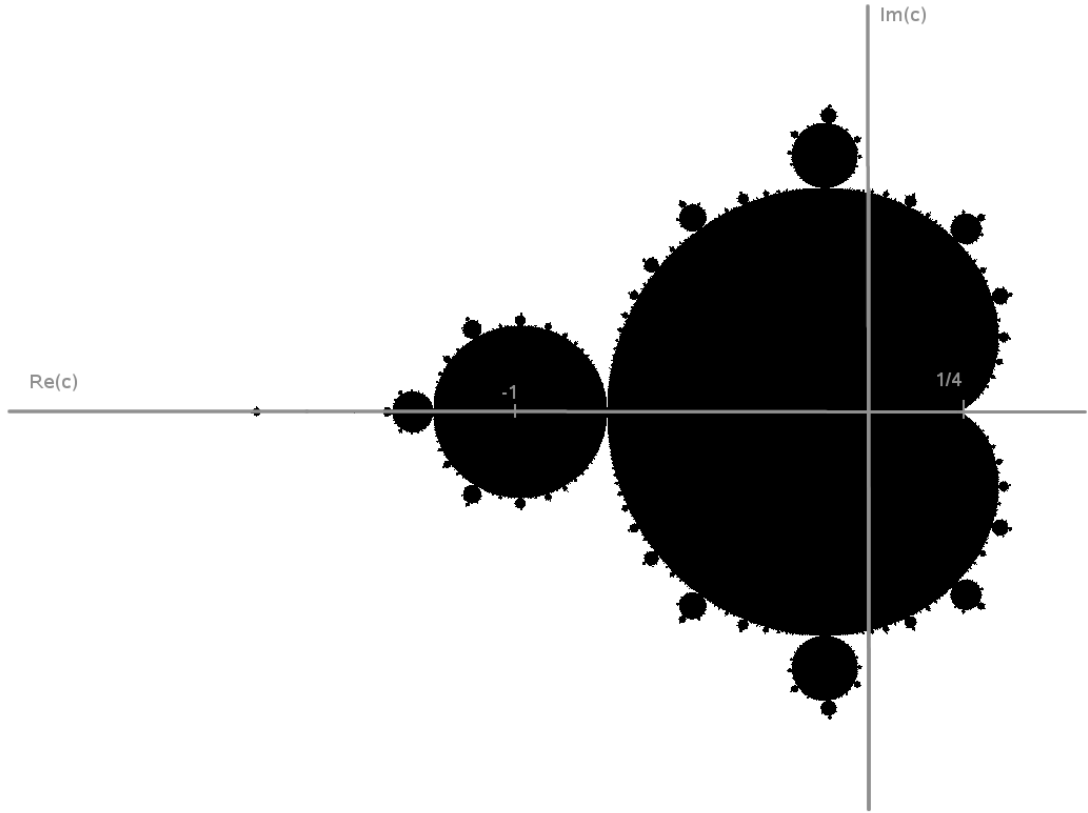}
  \caption{The Madelbrot set.}\label{M1}
  \label{fig:12}
\end{figure}

\end{Remark}
\section{Main results}

\subsection{Expansions at the fixed points}
Assume that $c$ is in the interior of $\mathcal{M}$
and let $$s=-\ln( e^{-2\pi i t} z)$$ Note that $s\to 0$ as $z\to e^{2\pi i
  t}$.
We can of course restrict the analysis to $t\in [0,1)$,
and from this point on we shall assume this is the case.
As is well known,  if $t=p/q$ with odd $q$, then the binary expansion of $t$ is
 periodic; in general  it is eventually periodic.

\begin{Theorem}\label{conj1}
  (i) Let $t=p/q$ with odd $q$, take $N$ between $1$ and $q-1$  so that $2^Nt=t \mod 1$, and let
  $M=2^N$.  There is a  $\ln M$-periodic function $\omega$, \footnote{The function $\omega$ depends, generally, on $t$.} analytic in the strip $\{\zeta:|\Im
  \zeta|<\pi/2\}$ and an entire function $g$ so that $g(0)=0,g'(0)=1$
  and
  \begin{equation}
    \label{eq:formvarphi}
    \varphi(z)=L_t+g\left(s^b\omega(\ln s)\right)
  \end{equation}
\begin{equation}
  \label{eq:valb}
  b=b(L_t)=\frac{\ln(P'_N(L_t))}{N\ln 2}
\end{equation}
for $|\arg\ln s|\le \pi/2$   and where $P_N(L_t)=L_t$. \footnote{It is interesting
to mention here  Eulers's totient theorem: if $n$ is a positive integer and $a$ is a positive integer coprime to $n$, then $a^{\phi(n)}=1 \mod n$, where $\phi(n)$
is the Euler totient function of $n$, the number of positive integers less than or equal to $n$ that are coprime to it. In our
problem, we need to solve the equation $(2^N-1)p=0 \mod q$ which
is implied by $(2^N-1)=0 \mod q$. This often allows for a good estimate on how large $N$ needs to be.} \footnote{If $t=0$  (\ref{eq:formvarphi}) applies with $N=1$  and  $L_0=0$
or $1=\lambda(1-2L_0)$, depending on $a$.}

(ii) If  $\lambda\notin\{0, 2\}$, then
the lines $\Im z=\pm \pi/2$ are natural boundaries for $\omega$. In particular,
$\omega$ is a nontrivial function for these $\lambda$.

(iii) If $t=p/(2^Mq)$ with $q$ odd and $M>0$, and $z=z_1 e^{2\pi i t}$,
then we have $z^{2^M}=z_1^{2^M}e^{2\pi i t'}$ where $t'$ is as in (i). We have
\begin{equation}
  \label{eq:other}
 \varphi(z)=L_t+\frac{g\left({s'}^b\omega(\ln s')\right)}{P'_M(L_t)} +g^2\left({s'}^b\omega(\ln s')\right)F_1\left(g\left({s'}^b\omega(\ln s')\right)\right)
\end{equation}
where $s'$ is as in (i) with $t$ replaced by $t'$ and $F_1$ is an algebraic
function, analytic at the origin.
\end{Theorem}
\begin{Corollary}\label{C1}{\rm 
   It follows from Theorem \ref{conj1} (i)
  that the Fourier coefficients $c_k$ of $\omega$ decrease roughly like $d^k$,
  with $d=e^{-2\pi^2/\ln M}$. Since $2\pi^2/\ln 2\approx 28.5$, $\omega$ can often be numerically replaced, with good accuracy, by a constant.}
\end{Corollary}
\begin{Corollary}\label{C2} {\rm The function $\phi$ has the following convergent transseries expansion near $z=e^{2\pi i t}$ ($t=p/q$, $q$ odd.)

\begin{multline}\label{trans}
  \phi(z)=L_t+\sum_{n=1}^{\infty}a_n \left(s^b \sum_{k=-\infty}^{\infty}c_k s^{2k\pi i/\ln{M}} \right)^n\\
  =L_t+\sum_{n=0}^{\infty} \sum_{k=-\infty}^{\infty} A_{n,k} (-\ln( e^{-2\pi i t} z))^{nb+2k\pi i/\ln{M}}
\end{multline}
where $a_n$ decrease faster than geometrically
and $A_{n,k}$ decrease 
faster than $\epsilon^n d^k$, with $d$ as in Corollary \ref{C1} and
$\epsilon>0$ arbitrary.
A similar result holds for $t=p/(2^mq)$.}
\end{Corollary}

\begin{proof}
A straightforward calculation using Theorem \ref{conj1}, where
$$g(z)=\sum_{n=1}^{\infty}a_nz^n,\; \omega(t)=\sum_{k=-\infty}^{\infty}c_ke^{2k\pi i t/\ln M}$$
The rate of decay of the coefficients follows immediately from Theorem
\ref{conj1} (i) and Corollary \ref{C1}.
\end{proof}

We note that, in some cases including
the interior of the main cardioid of $\mathcal{M}$, the expansion (\ref{trans}) converges on $\partial \mathbb{D}$ as well (though, of course, slower). This is a consequence of the
Dirichlet-Dini theorem and the H\"older continuity of
$\omega$ in the closure of its analyticity domain (by \eqref{eq:eqo})
and of the H\"older continuity of $\varphi$, shown, e.g., in \cite{CH}).

\begin{Note}{\rm Self-similarity is manifest in (\ref{eq:formvarphi}). Indeed, since
    $\omega$ is periodic and $g'\ne 0$ (see the proof of Lemma \ref{NFC}
    below), $\omega$ can be determined from any sufficiently large piece of
    $J$. Then, (\ref{eq:formvarphi}) shows that, up to conformal
    transformations and rescaling, this piece is reproduced in a
    neighborhood of any periodic point.}
\end{Note}

  We see that
\begin{equation}
  \label{eq:eqo}
  {\omega}(\ln s)= s^{-b}g^{-1}({\phi}(z)-L_t)
\end{equation}

\begin{Note}[Evaluating the transseries coefficients]\label{NN3} {\rm  There are many ways to
obtain the coefficients in \eqref{trans}. A natural way is the following.
(i) First, the series of $g$ is found by simply iterating the contractive map in Lemma \ref{NFC} below; the series
for $g^{-1}$ is calculated analogously.

(ii) The relation (\ref{eq:eqo}), together with the truncated Laurent
series of $\varphi$,   can be used
 over one period
of $\omega$ inside
the domain of convergence of the series
of $g^{-1}$, to determine a sufficient number of Fourier coefficients
of $\omega$. The numerically optimal period depends of course on the value of $c$.

The accuracy of (\ref{trans}) increases as the
boundary point is approached.
}

\end{Note}

\begin{Note}\label{N17}{\rm
The list below gives
 $q$ (in brackets), together with the  period of $1/q$   in base $2$
(underbraces).
\begin{equation}
  \label{eq:ordn}
  \underbrace{[3]}_2, \underbrace{[7]}_3, \underbrace{[5, 15]}_4, \underbrace{[31]}_5,
\underbrace{[9, 21, 63]}_6, \underbrace{[127]}_7,
\underbrace{[17, 51, 85, 255]}_8,...
\end{equation}
where more than one denominator indicates that $2^N-1$ is not prime,
and for each prime factor of $2^N-1$ we obviously get different
periodic orbits.  }
\end{Note}

\begin{Example}
  For $\lambda=0.9$ we get the following rounded off values of $b$  indexed by $N=2,3,...$ (note that cusps are generated
iff $\Re b<1$):
\begin{multline}
  \label{eq:angles}
  [0.13],[1.16],[1.08-0.145i, 1.08+0.15i],[0.98-0.19i, 0.98+0.19i, 1.09],\\
[0.904-0.21i, 0.904+0.21i, 1.04-0.069i, 1.04+0.069i, 1.12-0.089i, 1.12+0.089i]...
\end{multline}
with $\beta_t:=\Re \, b_t$ clearly given by
 \begin{equation}
  \label{eq:angles1}
  [0.13],[1.16],[1.08],[0.98, 1.09],
[0.904, 1.04, 1.12]...
\end{equation}
\end{Example}
\begin{Note}
  {\rm  Along the periodic orbit $L_t,P(L_t),...,P_{N-1}(L_t)$ we have $P'_N=const$
and $b=const$.
  Indeed, this follows from the fact that $P'_N(L_t)=P'(L_t)\cdots P'_{N-1}(L_t)$  is invariant
under cyclic permutations.}
\end{Note}

Part of Theorem \ref{conj1} follows from the more general result below. It
is convenient to map the problem to the right half plane, by writing $\varphi(z_0 e^{-t})=\varphi(z_0)+F_0(t)$.
\begin{Assumptions}
  (i) Let $A$ and $F_0$ be  analytic in the right half plane $\mathbb{H}$ and assume that for
some $n>1$
it satisfies the functional relation
\begin{equation}
  \label{eq:eqFnx}
  F_0(nx)=A(F_0(x))
\end{equation}
(ii) Assume that $F_0\to 0$ along any curve lying in a Stolz angle in $\mathbb{H}$
(nontangential limit, see
\cite{Conway}; this is the case for instance if $F_0$ is bounded near zero
and $F_0\to 0$ along some particular nontangential ray).

(iii) Assume that $|w|>1$, where $w=A'(0)$ (note that, by (i) and (ii), $A(0)=0$).
\end{Assumptions}

\begin{Theorem}\label{sc}
Under the assumptions above, there exists a unique  analytic function $g$,
with $g(0)=0$ and $g'(0)=1$, and a
multiplicatively periodic function $h$: $h(nx)=h(x)$, analytic in
  $\mathbb{H}$ so that, for sufficiently
small $x$, $F_0$ is of the form (see (\ref{eq:valb}))
\begin{equation}
  \label{eq:eqform}
  F_0(x)=g\left(x^{\log_n w}h(x)\right)
\end{equation}

Moreover, if $A$ is an entire function then $g$ is also an entire function, and the above expression is valid for all $x\in\mathbb{H}$.
  \end{Theorem}
\begin{Note}[Connection between transseries and local  angles] {\rm We see,
using Theorem \ref{conj1} (i), that in a neighborhood of a point
$t$ of period $M=2^N$, $J$ is the image of a small arc of a circle, or equivalently
 of a segment $[-\epsilon,\epsilon]$ under a transformation
of the form $\zeta\mapsto \zeta^{\beta+i\Im b}F(\zeta)$ ($\beta=\Re\,\, b$) with
$F$ multiplicatively periodic. In an averaged sense (over many periods),
or of course if $\Im\, b=0$ and $F$ is a constant, $\beta$ is the cusp
at $L$; in general, the shape is a spiral. }
\end{Note}
\begin{figure}
  \centering
  \includegraphics[scale=0.6]{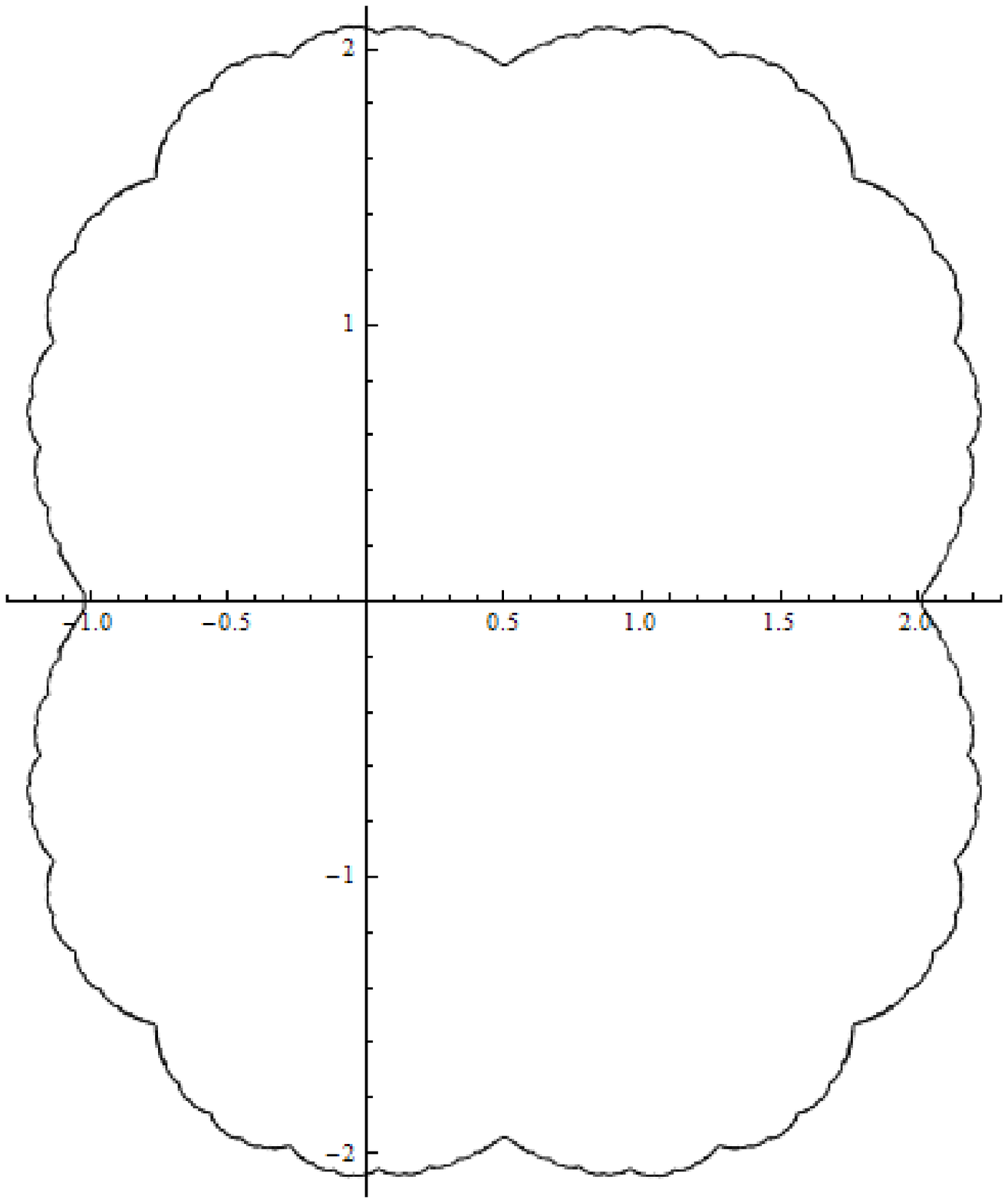}\includegraphics[scale=0.6]{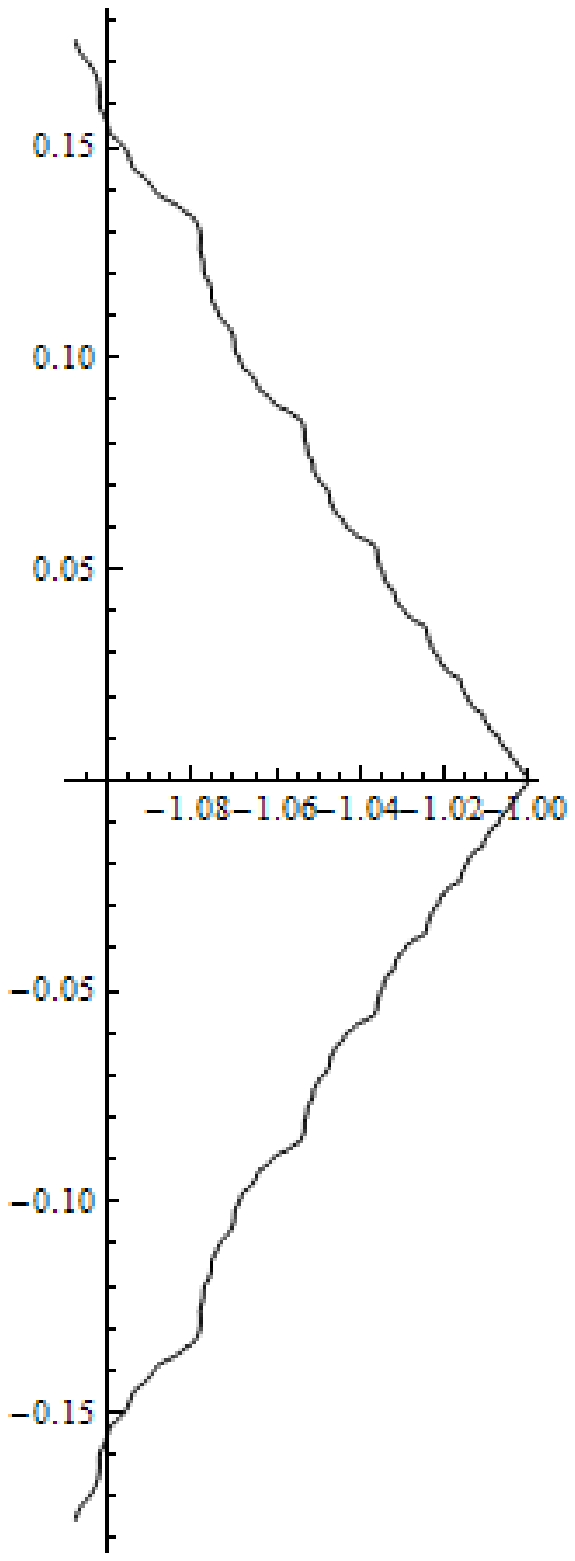}
  \caption{The Julia set for $\lambda=0.5$ (left figure), plotted by combining
rescaled ``bricks''. The ``brick'' (right): the local
shape obtained from  transseries expansion at $1$ .}\label{j1}
\end{figure}
\begin{figure}
  \centering
  \includegraphics[scale=0.6]{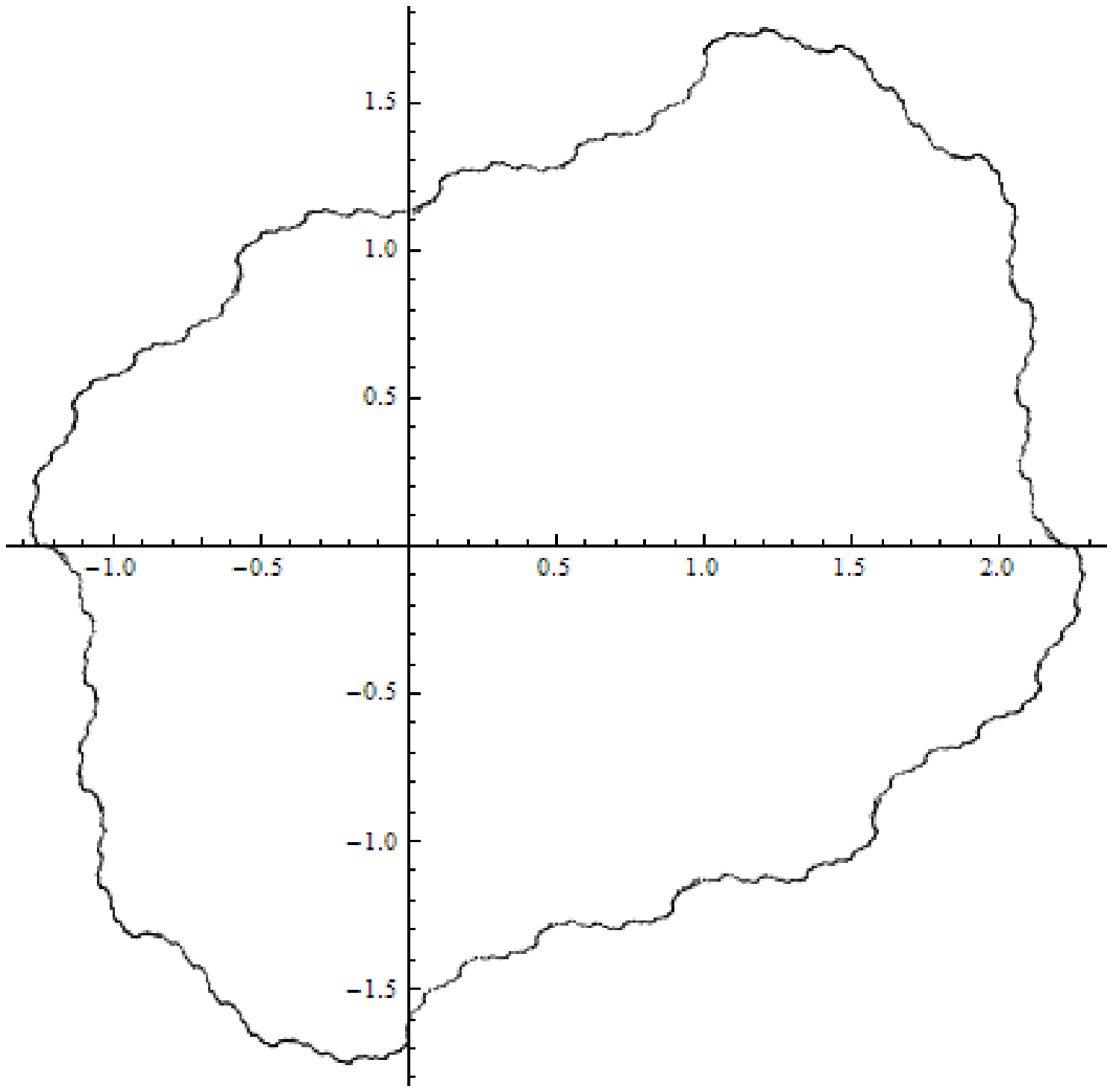}
  \caption{The Julia set for $\lambda=0.5i$, obtained from the transseries
as Fig. \ref{j1}.}\label{ji}
\end{figure}
\begin{figure}
  \centering
  \includegraphics[scale=0.6]{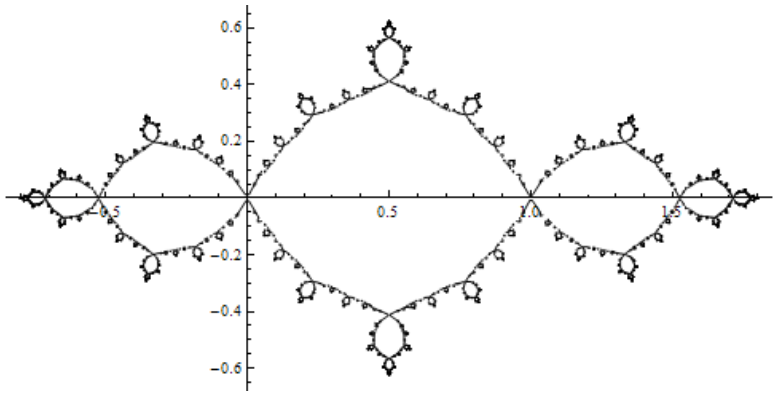}
  \caption{The Julia set for $\lambda=-1.25$, obtained from the transseries
as Fig. \ref{j1}.}\label{j2}
\end{figure}
\subsubsection{The average branching}
The critical point is outside $J$ (this is easy to show; see also the proof
of Proposition \ref{Bb}).  By  continuity, zero is outside $P'(J)$; by the argument principle,
$\Im\ln ( P'(J))$ is bounded by $2\pi$ and  $ P'(\varphi(e^{2\pi i x}))$
is bounded below and continuous.
Using  Proposition \ref{Bb} and the continuity of $\varphi$  on $\partial \mathbb{D}$ we see that the following holds.
\begin{Corollary}
  The average $b$,
  \begin{equation}
    \label{eq:eqbe1}
    b_E=\int_0^1\ln P'(\varphi(e^{2\pi i x})) dx
  \end{equation}
with the natural branch of the log, is well defined.
\end{Corollary}
Let $\beta_E=\Re \,\,b_E$.
\subsection{Recursive construction of $J$}
We say that a real number
is ``$(\epsilon,N,m)$-normal'' on the
initial set of $N$ bits if any block of bits
of length $m$ of itself and its $m$ binary left shifts ($2x\mod 1$)
appears with a relative frequency $1/Q$  within $\epsilon/Q$ errors, where $Q=2^m$.
Consider the set of numbers $\mathcal{N}_{N,m,\epsilon}$ in $[0,1]$ which are not
$(\epsilon,N,m)$-normal. The total measure Lebesgue measure of this set is estimated by,
see \S\ref{S3.1},
\begin{equation}
  \label{eq:nme}
  \text{ meas }( \mathcal{N}_{N,m,\epsilon})\le  2Qme^{-2N\epsilon^2/Q^2}
\end{equation}
Consequently, for large $N_0$, we have
\begin{equation}
  \label{eq:nme1}
  \text{ meas }\left(\mathcal{N}_{[N_0],m,\epsilon} \right)=O(Qme^{-2N_0\epsilon^2/Q^2})
\end{equation}
where
$$ \mathcal{N}_{ [N_0],m,\epsilon}=\bigcup_{N \ge N_0} \mathcal{N}_{N,m,\epsilon}$$
The complement set $ \mathcal{N}^c_{ [N_0],m,\epsilon}$ can be obtained by excluding
from $[0,1]$ intervals of size $2^{-Nm}$ around each binary rational
with $Nm$  bits,  which is not $(\epsilon,N,m)$-normal.

We denote as usual the Hausdorff dimension of $J$ by $D_H$.

\begin{Theorem}
  Consider the curve $\tilde{J}$ obtained from $J$ by eliminating the binary rationals,  modifying $\varphi$ in the following way. Define $\tilde{\varphi}=\varphi$ at all points $e^{2\pi i z}$
with $z\in \mathcal{N}^c_{ [N_0],m,\epsilon} $. On the excluded intervals, $\tilde{\varphi}$
is simply defined by linear endpoint interpolation (cf.
(\ref{eq:strtl})). Let $\tilde{J}=\tilde{\varphi}(\partial\mathbb{D})$. Then, for
any $\varepsilon>0$,

(i) The function $\tilde{\varphi}$ is H\"older continuous
of exponent at least $\beta_E-\varepsilon$.

(ii) The Hausdorff
dimension of the graph of $\tilde{\varphi}$ is less than
$2-\beta_E+\varepsilon$. Here $\beta_E\ge 1/2$, see Note \ref{Note15}.
\end{Theorem}
\begin{Note}
(i)
  {\em The Hausdorff dimension  of $\tilde{J}$ is lower than that
of $J$. Indeed, this follows from $D_H\ge 1/\beta_E$, see (\ref{eq:estD}) below, and
$$2-x < 1/x, \ \ \ x\in (0,1)$$

(ii) Also, in large sections of the Mandelbrot set,
including the main cardioid, the regularity of $\tilde{\varphi}$ is strictly
better than that of $\varphi$.

 In this sense, both the geometry (through regularity) and the Hausdorff
dimension come from rational angles (more precisely,
from the angles with non-normal distribution of digits
in base $2$).}

\end{Note}
\subsection{Hausdorff dimension versus angle distribution}

Through the Ruelle-Bowen formula we see that  $D_H$   can be seen as  ``inverse
temperature''\footnote{The terminology is motivated by the formula
  $\partial S/\partial E=T$, in units where $k_B=1$.}  of the cusp
system.

\z{\bf Notations} (See the proof of Proposition \ref{PP9} below for
more details.)  Let $\mu_n(\beta)=\text{prob}(\{z\in \text{fix}\, P_n:\beta(z)\le \beta\}$, where the probability is taken with respect to
the counting measure and let $F_n=\mu_n^{1/n}$. Let $w=P'(0)$. Note that $F_n\in
[0,1]$ are monotone (increasing) functions and right-continuous.  Define
$\overline{F}=\limsup _{n\to\infty}F_n$ and $\underline{F}=\liminf
_{n\to\infty}F_n$, and denote, as usual for monotone functions,
${F_{+}(x)}=F(x+ 0)$ (the function $F_+$ is clearly right
 continuous.)

Define $\Phi=-\log_{2} F\in [0,\infty]$ (similarly,
$\overline{\Phi}=-\log_2 \overline{F}$ etc.) and let $\Phi^\star(t)=\max(ts-\Phi(s))$\footnote{$\Phi^\star$ is the convex transform (Legendre transform if $\Phi$ is convex) of $\Phi$.}
\begin{Proposition}\label{PP9}
  We have
  \begin{equation}
    \label{eq:deftemp}
    \overline{\Phi}_+^\star(-D_H)= \underline{\Phi}_+^\star(-D_H)=-1
  \end{equation}
  \end{Proposition}
\begin{Note}
  {\rm Note also that all $b$ (cf. (\ref{eq:valb})) have nonnegative real part, since
    $|P'(L_t)|>1$ for $c$ in the hyperbolic components of  $\mathcal{M}$, as seen next}.
\end{Note}

\begin{Proposition}\label{Bb}
 In any  hyperbolic component of the Mandelbrot set there
is a $\delta>0$ so that for any $\tau\in J$  we have $|P'(\tau)|>\delta$.
\end{Proposition}
\begin{proof}
  It is known \cite{Beardon} p. 194 that the immediate basin of any
 attracting cycle contains at least one critical point. Therefore, in
the hyperbolic components of $\mathcal{M}$ the critical
point cannot be on $J$.
 Since
\begin{equation}
  \label{eq:estimpship}
  P'(\varphi(z))\varphi'(z)=2z\varphi'(z^2)
\end{equation}
the critical point cannot be inside either,  since
otherwise, solving (\ref{eq:estimpship}) for $\varphi'(z^2)$ in terms
of $\varphi'(z)$, it is clear that $\varphi'$  would vanish on a set with
an accumulation point at $z=0$.
Therefore  $|P'|>0$
on the continuous curve $J$.
\end{proof}

\begin{Theorem}  Assume $c$ is in a hyperbolic component of $\mathcal{M}$.
(i)   The Hausdorff dimension $D_H$ of $J$ satisfies
  \begin{equation}
    \label{eq:dimH}
    D_H\ge \beta_E^{-1}
  \end{equation}

(ii) On a set of full measure, $\varphi$ is H\"older continuous with
exponent at least $\beta_E-\epsilon\ge 1/2-\epsilon$
for any $\epsilon>0$.

\end{Theorem}
A direct  and elementary
proof of the theorem is given in \S\ref{S3.1}.
\begin{Note}\label{Note15}
  Since $D_H\le 2$, it follows that $\beta_E\ge 1/2$. (By a fundamental result
of Shishikura \cite{Shishikura}, $D_H=2$ on the boundary
of $\mathcal{M}$.)
\end{Note}
\section{Proofs and further results}

\begin{proof}[Proof of Theorem~\ref{sc}]
Note that
  \begin{equation}
  \label{eq:eqP3}
  A(y)=wy+y^2A_0(y)\ \ \text{where} \ \ A_0(y)=y^{-2}(A(y)-wy)\ \text{is a polynomial.}
\end{equation}
We use an analytic solution of (\ref{eq:eqFnx}) to bring the equation
of $F_0$ to a normal form.
\begin{Lemma}[Normal form coordinates]\label{NFC} There is a unique function $g$ analytic
in a disk $\mathbb{D}_{\epsilon}$, such that $g(0)=0, g'(0)=1$  (thus analytically invertible near zero) and
\begin{equation}
  \label{eq:eqg}
  g(wy)=A(g(y))
\end{equation}
\end{Lemma}
\begin{proof}
  We write $g(y)=y+y^2 g_0(y)$, $\alpha=1/w$ and get
\begin{equation}
  \label{eq:eq4}
  g_0(y)=\alpha g_0(\alpha  y)+\alpha^2(1+\alpha y g^2_0(\alpha y))^2A_0\left(\alpha y
    +\alpha^2 y^2 g_0(\alpha y)\right)
\end{equation}
A straightforward verification shows that, for small $\epsilon$, (\ref{eq:eq4}) is contractive in the
space of analytic functions in $\mathbb{D}_{\epsilon}$ in the ball $\|g_0\|\le
2|A_0(0)|$, in the sup norm.

Define $H(x)=g^{-1}(F_0(x))$. (The definition is correct for small  $x$ since
$g$ is invertible for small argument, and $F_0$, by assumption is small). Obviously $H$ is analytic for small $x$.
We see that
\begin{equation}
  \label{eq:eqh7}
 H(nx)=g^{-1}(A(F_0(x)))=g^{-1}(A(g(H(x))))=g^{-1}(g(wH(x)))=wH(x)
\end{equation}
by (\ref{eq:eqg}). Taking $h(x)=x^{-\log_n w}H(x)$, the conclusion
follows. Note that for any $r$, if $g$ is analytic in $\mathbb{D}_r,$ then, by (\ref{eq:eqg}) and
the monodromy theorem, $g$ is analytic in $\mathbb{D}_{|w|r}$ as long as $A$ is analytic in $\mathbb{D}_r$; since $r$ is
arbitrary, it follows that $g$ is entire if $A$ is entire. In the same way, since $h(nx)=h(x)$,
$h$ is analytic in $\mathbb{H}$. Note also that $g'$ is never zero, since
otherwise it would be zero on a set with an accumulation point at $0$, as it is seen by an argument similar to the one in the paragraph following (\ref{eq:estimpship}).
 \end{proof}

\subsection{Probability distribution of angles}\label{S3.1}
Consider the periodic points of period $mN$ ($m$ and $N$ conveniently large).
These correspond, through $\varphi^{-1}$, to points of the form $z_t =e^{2\pi i t}$
where $t$ has a  periodic binary  expansion of period $mN$.

Consider the orbit $z_t, z_t^2,...,z_t^{2^{Mn-1}}$ (by definition,
$z_t^{2^{Mn}}=z_t$).  We have, by formula (\ref{eq:valb}), with
$L_t=\varphi(z_t)$,
\begin{equation}
  \label{eq:valb2}
  b(L_t)=N^{-1}m^{-1}\sum_{j=0}^{Nm-1}\log_2[P'(\varphi(z_t^{2^{j}}))]
\end{equation} We analyze the deviations from  uniform  distribution
of subsequences of $m$ consecutive bits in the block of length $Nm$. For
this, it is convenient to rewrite the  block of length $Nm$ in base $Q=2^m$, as now a block of length $N$ of $Q$-digits. Every binary $m$-block corresponds to a digit in $\{0,1,...,Q-1\}$ in base $Q$.
To analyze the deviations, we  rephrase the question as follows. Consider $N$ independent variables, $X_1, ...,X_N$ with values: $1$ with probability $1/Q$ if the digit $i$ equals $q$,
and $0$ otherwise. The expectation  $E(N^{-1} (X_1+\cdots+ X_N))$ is
clearly $1/Q$ and  we have $\mathcal{P}(X_i-E(X_i))\in[-1/Q,1-1/Q]=1$ ($\mathcal{P}$ denotes probability). Then,  with
$S=X_1+\cdots+X_N$ we have, by
Hoeffding's inequality \cite{Hoeffding},
\begin{equation}
  \label{eq:Hoeff}
  \mathcal{P}(|N^{-1}S-1/Q|>\epsilon/Q)\le 2e^{-2N\epsilon^2/Q^2}
\end{equation}
Using the elementary fact that $\mathcal{P}(A\vee B)\le \mathcal{P}(A)+\mathcal{P}(B)$, we see that
the probability of a block of length $N$ having the frequency
of any digit departing $1/Q$ by $\epsilon/Q$ is at most
\begin{equation}
  \label{eq:pepsi}
  \mathcal{P}_{\epsilon}\le  2Qe^{-2N\epsilon^2/Q^2}
\end{equation}
We see that (\ref{eq:valb2}) involves shifts in base $2$ (and not
in base $2^m$).  The probability of a block of length $Nm$ in base $2$ having
the frequency of any $m$-block in all its  $m$ successive binary left-shifts ($x\to 2x \mod 1$) departing by $\epsilon/Q$
from its expected frequency of $1/Q$ is thus
\begin{equation}
  \label{eq:Pepsi}
  \mathcal{P}\le  2Qme^{-2N\epsilon^2/Q^2}
\end{equation}
Therefore, the relative frequency of ``$\epsilon$-normally distributed'' $Nm$-periodic binary
expansions with all $m$-size blocks of its binary shifts
distributed within $\epsilon/Q$ of their expected average number is
\begin{equation}
  \label{eq:Pepsi1}
  \mathcal{P}\ge 1-  2Qme^{-2N\epsilon^2/Q^2}
\end{equation}
Let
\begin{equation}
  \label{eq:be}
  b_{EQ}=Q^{-1}\sum_{j=0}^{Q-1} \log_2[P'(\varphi(e^{2\pi i j/Q}))]
\end{equation}
We take $f_1=\Re \log_2 \left(P'\circ\varphi\right) $ and $f_2=\Im \left(\log_2 P'\circ\varphi\right)$,
and for a real function $f$ we write $f^+$ for its positive part and $f^-$
for its negative part. For any number $t$ which is $\epsilon$-normally distributed,
the sequence $2^j t \mod 1$ will have $Nm(1/Q\pm \epsilon/Q)$ points
in each interval of the form $[j/Q, j+1/Q]$. Therefore, taking
the positive real part of the integrand in (\ref{eq:valb2}), we have
the following bound for its contribution to the sum:
\begin{multline}
  \label{eq:sum}
Q^{-1} (1-\epsilon) \sum_{j=0}^{Q-1} \min_{x\in [j/Q,(j+1)/Q]} f_1^+(e^{2\pi i x})\\ \le Q^{-1}\sum_{j=0}^{Q-1}\rho(j/Q) \min_{x\in [j/Q,(j+1)/Q]} f_1^+(e^{2\pi i x})\le   N^{-1}m^{-1}\sum_{j=0}^{2^{Nm}-1}f_1^+(z^{2^j})
\end{multline}
where $\rho(j/Q)$ is the frequency of $2^j t \mod 1$ belonging
to $[j/Q,(j+1)/Q]$. Corresponding estimates hold with $\le$ replaced by $\ge$ and $\min$ with $\max$.
Since
\begin{equation}
  \label{eq:Riem}
  Q^{-1} (1-\epsilon) \sum_{j=0}^{Q-1} \min_{x\in [j/Q,(j+1)/Q]} f_1^+(e^{2\pi i x})\to
\int_0^1f_1^+(e^{2\pi i x})dx
\end{equation}
as $Q\to\infty$ and $\epsilon\to 0$ (and similarly for $f_1^-$ and $f_2^{\pm}$), for any
$\epsilon_1>0$ we can choose $Q$ large enough and and $\epsilon$ small
enough so that on the set of blocks described above (\ref{eq:Pepsi})
we have
\begin{equation}
  \label{eq:b-be}
  |b_{EQ}-b_E|<\epsilon_1
\end{equation}
Clearly then, we have
\begin{multline}
  \label{eq:eqmun}
  1\ge \left(\mu_{NM}(b_E+\epsilon)-\mu_{NM}(b_E-\epsilon)\right)^{1/{Nm}}\\ \ge \left(1-  2Qme^{-2N\epsilon^2/Q^2}\right)^{1/{Nm}}\to 1\ \ \text{as } N\to \infty
\end{multline}
and {\em a fortiori}
\begin{equation}
  \label{eq:eqmu3}
 \mu_{Nm}^{1/{Nm}}(b_E+\epsilon)\to 1\ \ \text{as } N\to \infty
\end{equation}
Therefore,
\begin{equation}
  \label{eq:eqF2}
  \overline{\Phi}(\beta_E+\epsilon_1)=0
\end{equation}
for all $\epsilon_1>0$ and hence $\overline{\Phi}_+(\beta_E)=0$.

On the other hand,
\begin{equation}
  \label{eq:star1}
 -1= \max\left(-t D_H - \overline{\Phi}_+(t)\right)\ge -\beta_E D_H -\overline{\Phi}_+(\beta_E)=-\beta_E D_H
\end{equation}
and thus
\begin{equation}
  \label{eq:estD}
  D_H\ge 1/\beta_E
\end{equation}
\begin{Note}{\rm
  Another approach to obtain  (\ref{eq:dimH}) is the following, using the general
  Ruelle-Manning formula, cf.  \cite{Ruelle4} p. 344, which can be written in the form
\begin{equation}
  \label{eq:ruelle-manning}
  D_H=\sup_{\mu} \left(e_{\mu}(q)/\int_J \log |P'|d\mu \right)
\end{equation}
where $\mu$ is a $P-$ invariant measure and $e_{\mu}(P)$ is the
entropy of $P$ with respect to $\mu$. An inequality obviously follows
by choosing any particular invariant measure. The measure in using
(\ref{eq:ruelle-manning}) to derive the inequality would be
$d\mu=\varphi^{-1}(dx)$ where $dx$ is the Lebesgue measure on $[0,1]$
and the inequality would follow by estimating $e_{\mu}(P)$. }
\end{Note}
\subsection{H\"older continuity on a large measure set}
\begin{proof}
We obtain, from (\ref{eq:estimpship}),
\begin{equation}
  \label{eq:qq}
   \varphi'(z)=z^{2^M}\frac{2^M}{\prod_{j=0}^{M-1}P'(z^{2^j})}\varphi'(z^{2^M})
\end{equation}
Let $\zeta\in \mathcal{N}^c_{\ge [N_0],m,\epsilon}$. We  let $\rho=1-2^{-M}\epsilon_3$ where $M$ will be chosen large.
By the continuity of $\log P'(\varphi)$, for any $\delta>0$ we can choose an $\epsilon_3$ small enough so that for any large  $M$ we have
\begin{equation}
  \label{eq:eqint3}
 \left| M^{-1}\sum_{j=1}^M\log_2  P'(\varphi(\zeta^{2^j}\rho^{2^j}))- M^{-1}\sum_{j=1}^M\log_2  P'(\varphi(\zeta^{2^j}))\right|<\delta/2
\end{equation}
On the other hand, we can choose $M$ large enough so that, reasoning as for (\ref{eq:sum}), we get
\begin{equation}
  \label{eq:eqint32}
 \left| M^{-1}\sum_{j=1}^M\log_2  P'(\varphi(\zeta^{2^j}))-b_E\right|<\delta/2
\end{equation}
For any $\delta_1>0$ we can choose $\epsilon_3$ small enough so that
in turn $\rho^{2^M}$ is sufficiently close to one so that
\begin{equation}
  \label{eq:eq41}
 \left|\prod_{j=1}^M P'(\varphi(\zeta^{2^j}\rho ^{2^j}))\right|\ge 2^{M\beta_E-M\delta_2}
\end{equation}
where $\delta_2=\delta+\delta_1$.
We write $\zeta\rho=\zeta-dx$ and note that $2^M=\epsilon_3/ dx$. Taking $z=\zeta\rho$, we obtain, combining (\ref{eq:qq}), (\ref{eq:eqint3}) (\ref{eq:eqint32}) and (\ref{eq:eq41}),
\begin{equation}
  \label{eq:estimpship1}
 |\varphi'(\zeta-dx)|\le \left|\frac{dx}{\epsilon_3}\right|^{-1+\beta_E-\delta_2}\max_{|z|=1-\epsilon_3}|\phi'(z)|
\end{equation}
or, for some absolute constant $C$,
\begin{equation}
  \label{eq:phip}
  |\phi'(\zeta-dx)|\le C |dx|^{\beta_E-\delta_2-1}
\end{equation}

Thus, by integration, for any two points $x_{1,2}$ in $\mathcal{N}_{[N_0],m,\epsilon}$, we have, for an absolute constant $C_1$,
\begin{equation}
  \label{eq:eqHoldav}
   |\varphi (x_1)-\varphi (x_2)|\le C_1|x_1-x_2|^{\beta_E-\delta_2}
\end{equation}
For $x$ in an (entire) excluded interval $[x_1,x_2]$ in the construction
of $ \mathcal{N}^c_{[N_0],m,\epsilon}$ we replace the curve $x\mapsto
\varphi(e^{2\pi i x})$ by the straight line
\begin{equation}
  \label{eq:strtl}
 \tilde{\varphi}=x\mapsto  \frac{x_2-x}{x_2-x_1}\varphi(x_1)+\frac{x_1-x}{x_2-x_1}\varphi(x_2)
\end{equation}
and let $\tilde{\varphi}=\varphi$ otherwise. The new curve  $\tilde{\varphi}$
is clearly H\"older continuous of exponent $\beta_E-\delta_2$. Indeed,
we can use the inequality
\begin{equation}
  \label{eq:estm4}
  \frac{1+x^{\lambda}}{(1+x)^{\lambda}}\le 2^{1-\lambda} \text{ for } x \text { and } \lambda \text{ in } (0,1)
\end{equation}
to check that
\begin{equation}
  \label{eq:est5}
  \frac{|x_1-x_2|^{\lambda}+|x_2-x_3|^{\lambda}}{|x_1-x_2+x_2-x_3|^{\lambda}}\le 2^{1-\lambda}\text{ for } x_1<x_2<x_3 \text { and } \lambda \text{ in } (0,1)
\end{equation}
For $x<y$ in an interval  $[a,b]$ where $\tilde{\varphi}$ is a straight line,
with $\tilde{\varphi}(a)=X,\tilde{\varphi}(b)=Y$ and $t, s$ in $(0,1)$
we have
\begin{multline}
  \label{eq:phitilde}
  \frac{\tilde{\varphi}(x)-\tilde{\varphi}(y)}{(x-y)^{\lambda}}=
\frac{(t X+(1-t) Y -(s X+(1-s )Y)}{[t a+(1-t) b -(s a+(1-s) b)]^\lambda}\\=
\frac{(t-s)(X-Y)}{(t-s)^\lambda(b-a)^\lambda}\le C(t-s)^{1-\lambda}\le C
\end{multline}
H\"older continuity follows from (\ref{eq:eqHoldav}), (\ref{eq:phitilde}) and
the ``triangle-type'' inequality (\ref{eq:est5}).

The statement about the Hausdorff dimension
follows from the H\"older exponent,  see \cite{Urbansky} p. 156 and p. 168, implying that
the Hausdorff
dimension of the graph of $\tilde{\varphi}$ is less than
$2-\beta_E+\epsilon$.
\end{proof}
\end{proof}

  \subsection{Calculation of the transseries at rational angles. Proof of Theorem~\ref{conj1}}\label{PT1}
\begin{Note}
  {\rm By (\ref{eq:eqH}), we  have
    \begin{equation}
      \label{eq:tt}
      \varphi\left(z e^{2\pi i t}\right)=L_t+\mu\left(z e^{2\pi i t}\right)
    \end{equation}
where
\begin{equation}
  \label{eq:mu}
  \mu\left(z e^{2\pi i t}\right)\to 0 {\text{ as  }} z \to 1  \ \ \text{nontangentially}
\end{equation}
 }\end{Note}
\begin{Note}{\rm We can of course restrict the analysis to $t\in [0,1)$,
and from now on we shall assume this is the case.

}\end{Note}
\bigskip

 From this point on we shall assume that $t\in [0,1)$ has a periodic binary expansion.

 \begin{Note}\label{defM}
   We let $N_t$ be the smallest $N> 0$ with the property that $2^N t=t \mod 1$.
Let $P:=P_N$.  $P_N$ is a polynomial of degree $M=2^N$.
 \end{Note}

\begin{Note}\label{per}{\rm
 By (\ref{eq:eqH}) we have
  \begin{equation}
    \label{eq:lim}
   P( L_t)=L_t
  \end{equation}
and $L_t$ is a periodic point of $f$ (this, in fact, is instrumental
in the delicate analysis of \cite{DouadyHubbard}). Also, we have
\begin{equation}
  \label{eq:eqPP}
  P_N(\varphi(z e^{2\pi i t}))=\varphi(z^{2^N}e^{2\pi i t})
\end{equation}
}\end{Note}
\begin{Note}{\rm
  Since the Julia set is the closure of unstable periodic points, by Note~\ref{per} we must have
  \begin{equation}
    \label{eq:evalder}
P'(L_t):=w=1/\alpha \Rightarrow |w|\ge 1
  \end{equation}
\begin{Proposition}[See \cite{McMullen}, p.61]\label{KM}
  For the quadratic map, if $f$ has an indifferent cycle, then $c$ lies in the
  boundary of the Mandelbrot set.
\end{Proposition}

By Proposition~\ref{KM},  in our assumption on $c$ and since hyperbolic components belong to the interior of $\mathcal{M}$, we must have
\begin{equation}
  \label{eq:alphaval}
  |w|> 1
\end{equation}
}
\end{Note}

\begin{proof}[Proof of Theorem~\ref{conj1}] (i)
    Let $F_0(x)=\varphi(e^{2\pi i t-x})-L_t$ and $A(y)=P_N(y+L_t)-L_t$. The
    statement
now follows from Theorem~\ref{sc}  with $h(s)=\omega(e^s)$.

Note that $h$ cannot be constant,
or else $e^{2\pi i t}$ would be a point near which analytic continuation
past $\mathbb{D}$ would exist, contradicting Theorem \ref{conj1}, (ii).

(ii) Note that if $\varphi$ is analytic at some
binary rational, then it is analytic at one, since
\begin{equation}
  \label{eq:eqJ}
 \varphi(z^{2^J})= P_J(\varphi(z))
\end{equation}
On the other hand, $\varphi(1)=P(\varphi(1))$ and thus either
$\varphi(1)=0$ (possible if $|\lambda|>1$) or $\varphi(1)=\lambda^{-1}(\lambda-1)$
(possible if $|2-\lambda|>1$).
For $\varphi$ to be analytic at one, we must have $b_1\in\NN$,
or $P'=2^k$, $k\in\NN$. This means $\lambda=2^n,n\in\NN$ or $\lambda=2-2^n,n\in\NN$
and, to have  $c\in \text{int}\mathcal{M}$ we see that the only possibilities
are $\lambda\in \{0,2\}$.
  \end{proof}

\subsection{Proof of Proposition~\ref{PP9}}
\begin{proof} We only prove the result for $\overline{\Phi}_+$, since
the proof for $\underline{\Phi}_+$ is very similar.

Note first that $\max \{\beta\in B_n:n\in\NN\}<\beta_M<\infty$. (Indeed,
since the Julia set is compact, we have  $\|P'\|_{\infty,J}<K<\infty$, and thus $|P'_N|<K^N$ for some $K$.)

  We start from  Ruelle-Bowen's implicit relation for the Hausdorff dimension $D_H$ \cite{Ruelle},
\begin{equation}
  \label{eq:Ruelle1}
 \lim_{n\to\infty} A_n(D_H)= \lim_{n\to\infty}\sum_{z\in \text{fix}(P_n)}|P'_n(z)|^{-{D_H}}=1
\end{equation}
With $B_n=\{\beta(z): z\in \text{fix}(P_n)\}$, $\alpha=2^{D_H}$ we then have (see (\ref{eq:valb})
and Note \ref{N17})
\begin{equation}
   \label{eq:Ruelle2}
 \lim_{n\to\infty}\sum_{\beta \in B_n }\alpha^{-n \beta }N_n(\beta)=
 \lim_{n\to\infty}\sum_{\beta \in B_n }2^n\alpha^{-n \beta}\rho_n(\beta)=1
\end{equation}
where $N_n(\beta)$ is the degeneracy of the value $\beta$ and
$\rho_n(\beta)$ is the (counting) probability of the value $\beta$
within $B_n$.  Denote as usual by $\delta$ the Dirac mass at zero. We get,
for any $\epsilon>0$ (integrating by parts  and noting that
$\mu_n(s)\alpha^{-s}=0$ at $-\epsilon$ and at infinity),
\begin{multline}
  \label{eq:dangl}
2^{-n}A_n(D_H)=
 \int_{-\epsilon}^{\beta_M}d\beta\alpha^{-n\beta }\sum_{\beta'\in B_n}\rho(\beta')\delta(\beta-\beta')\\D_H\ln 2\int_{-\epsilon}^{\infty} \mu_n(s)\alpha^{-ns}ds
=:nD_H\ln 2\int_{0}^{\infty} F^n_n(s)\alpha^{-ns}ds
\end{multline}

We first estimate away  the integral from $\beta_M$ to infinity. Since $\overline{F}(t)=1$ for $t>\beta_M$, we have
\begin{equation}
  \label{eq:eqest4}
  \int_{\beta_M}^{\infty}\mu_{n_k}(s)\alpha^{-s}ds= \int_{\beta_M}^{\infty}\alpha^{-s}ds=o(e^{-\alpha \beta_M})
\end{equation}
Since $\beta_M$ can be chosen arbitrarily large, this part of the integral
does not contribute to the final result. Therefore we only need to show that
$$\lim_{n\to\infty}\left(\int_{0}^{\beta_M} F^n_n(s)\alpha^{-ns}ds\right)^{1/n}=\max_{s\in [0,\beta_M]} \overline{F}(s)\alpha^{-s}$$
since according to (\ref{eq:dangl})
$$\lim_{n\to\infty}\log_2 \left(\int_{0}^{\infty} F^n_n(s)\alpha^{-ns}ds\right)^{1/n}=-1$$

\begin{Proposition}
  Let $f:[a,b]\to [0,1]$ ($0\le a<b<\infty$) be increasing. Assume
  further that $f\equiv 1$ on $(b',b)$ where $b'<b$. Then, if $\alpha>1$, we have  $$\sup
  f_+(s)\alpha^{-s}=\max f_+(s)\alpha^{-s}=f_+(m)\alpha^{-m}$$ for some, possibly
  non-unique, $m\in[0,b']$.

\end{Proposition}
\begin{proof}
  The proof is elementary and  straightforward.
\end{proof}

Consider a countable dense set $S$ and
for each $s\in S$ take a subsequence $\{F_{n;s}\}$ so that $F_{n;s}\to \overline{F}(s)$
as $n\to\infty$. By a diagonal argument we find a subsequence
$\{F_{n_k}\}$ converging to $\overline{F}$  on $S$.  By abuse of notation,
we call this sequence $F_n$.

By standard results
on sequences of monotone functions,  \cite{Doob} p. 165, $\{F_{n}\}_{n\in\NN}$
converges to $\overline{F}$ at all points of continuity of $\overline{F}$, that is
on $[0,\beta_M]$ except for a countable set, and
the convergence is uniform on any interval of continuity
of $F$.
\begin{Proposition}\label{PP11}
  Assume that $f:[a,\infty)\to[0,1]$ is increasing and right continuous ($f=f_+)$. Let $m$ be a point of maximum of
$f(x)\alpha^{-x}$. Then,

(i) For all $x>0$
we have
\begin{equation}
  \label{eq:eqest5}
  |f(m+x)\alpha^{-m-x}-f(m)\alpha^{-m}|\le f(m)\alpha^{-m}(1-\alpha^{-x})\le x \ln\alpha
\end{equation}

(In particular  $f$ is H\"older right-continuous  at $m$,  with exponent one.)

(ii) We have  $\sup f=\max f=\text{\rm essup}f$.
\end{Proposition}\label{P11}

\begin{proof}
(i)  Using monotonicity and the definition of $m$ we have, for all $x>0$,
\begin{equation}
  \label{eq:eqfn}
  f(m)\alpha^{-m-x}\le f(m+x)\alpha^{-m-x}\le f(m)\alpha^{-m}  \end{equation}
which implies (\ref{eq:eqest5}).

(ii) This is a straightforward consequence of (i).
\end{proof}
\z Using Proposition \ref{PP11} (i), with the notations there, we see that
\begin{equation}
  \label{eq:est4}
\epsilon^{\frac{1}{n}}(1-2\epsilon\ln\alpha)    \max(F_n(x)\alpha^{-x})
             \le\left(\int_{0}^{{\beta_M}} dt F_n^n\alpha^{-nt}\right)^{\frac{1}{n}} \le \beta_M^{\frac{1}{n}}\max(F_n (x)\alpha^{-x})
\end{equation}
for all $\epsilon>0$. Thus we only need to show $\max(F_n (x)\alpha^{-x})- \max(F (x)\alpha^{-x})\to 0$ as $n\to\infty$.
Proposition \ref{PP9}  follows using (\ref{eq:est4}), Proposition \ref{PP11} (ii) and the following lemma.
\begin{Lemma}\label{Lessup}
  Assume $\sup_{[a,b]}\|f_n\|_{\infty}\le 1$ and $f_n\to f$ pointwise a.e. on $[a,b]$. Assume further that meas$\{x:f_n(x)>  \text{\rm essup}_{[a,b]} f_n-\epsilon\}>c(\epsilon)>0$ (uniformly in $n$) for all $\epsilon>0$. Then
  \begin{equation}\label{essup}
    \text{\rm essup}_{[a,b]} f_n \to  \text{\rm essup}_{[a,b]} f
  \end{equation}
\end{Lemma}
\begin{proof}
  This is standard measure theory; it  follows easily, for
  instance, from the definition of essup and Egorov's theorem.
\end{proof}
\end{proof}

\subsection{Proof of B\"ottcher's theorem}\label{S5}
({\bf Note:} this argument extends to general analytic maps.)

 We
write $\psi=\lambda z+\lambda^2zg(z)$ and obtain
\begin{equation}
  \label{eq:H}
  g(z)-\frac{1}{2}g(z^2)=\frac{1}{2}z+\frac{1}{2}\lambda\left[g(z)(z-g(z))+g(z^2)\right]+\frac{\lambda^2 z}{2}g(z)g(z^2)=N(g)
\end{equation}
 We define  the linear operator $\mathfrak T=\mathfrak T_2$, on $\mathcal A(\mathbb{D})$ by
\begin{equation}
  \label{eq:defT}
  (\mathfrak T f)(z)=\frac{1}{2}\sum_{k=0}^{\infty}2^{-k}f(z^{2^k})
\end{equation}
This is the inverse of the operator $f\mapsto 2f-f^{\vee 2}$,
where $f^{\vee p}(z)=f(z^p)$.
Clearly, $\mathfrak T f$ is an isometry on   $\mathcal A(\mathbb{D})$
and it maps simple functions, such as generic polynomials, to
functions having $\partial \mathbb{D}$ as a natural boundary; it reproduces $f$
across vanishingly small scales.

We write (\ref{eq:H}) in the form
\begin{equation}
  \label{eq:eqH4}
  g=2\mathfrak TN(g)
\end{equation}
This equation is manifestly contractive in the sup norm, in the ball
of radius $1/2+1/4$ in $\mathcal{A_{\lambda}}$, the functions analytic
in the polydisk $\mathbb{P}_{1,\epsilon}=\mathbb{D}\times
\{\lambda:|\lambda|<\epsilon\}$, if $\epsilon$ is small enough.   For $\lambda\ne 0$, $\varphi=\psi^{-1}$ is analytic for small $z$ as well..

\section{Acknowledgments} 

Work supported by in part by NSF grants
 DMS-0406193 and DMS-0600369. Any
opinions, findings, conclusions or recommendations expressed in this material
are those of the authors and do not necessarily reflect the views of the
National Science Foundation.

\end{document}